\documentclass{amsart}

\usepackage{amsmath}
\usepackage{amssymb}
\usepackage{amsfonts}
\usepackage{mathrsfs}

\usepackage[latin1]{inputenc}

\newcommand{\C}{{\mathbf C}}

\newcommand{\D}{{\mathbf D}}
\newcommand{\s}{{\mathbf S}}
\newcommand{\Z}{{\mathbf Z}}

\newcommand {\hk} {\mathscr {H}_k^N}

\newcommand{\B}{{\mathbf B}}
\newcommand{\N}{{\mathbf N}}
\newcommand{\R}{{\mathbf R}}

\newcommand{\beq}{\begin{equation}}
\newcommand{\eeq}{\end{equation}}
\newcommand{\beqs}{\begin{equation*}}
\newcommand{\eeqs}{\end{equation*}}
\newcommand{\ba}{\begin{eqnarray}}
\newcommand{\ea}{\end{eqnarray}}
\newcommand{\bas}{\begin{eqnarray*}}
\newcommand{\eas}{\end{eqnarray*}}
\newcommand{\bg}{\begin{gather}}
\newcommand{\eg}{\end{gather}}
\newcommand{\bgs}{\begin{gather*}}
\newcommand{\egs}{\end{gather*}}

\newcommand{\Sm}{S}

\newcommand{\ta}{\theta}
\newcommand{\infinity}{\infty}

\newcommand {\hv} {h^\infty_g}
\newcommand {\hi} {h^\infinity} 
\newcommand {\hg} {\hv} 

\newcommand {\hvt} {h^\infty_{\tilde{g}}}
\newcommand {\hgt} {\hvt}

\newcommand{\ve}{\varepsilon}
\newcommand{\gt}{\tilde{g}}

\newcommand{\la}{\lambda}

\newcommand{\fraction}{\frac}

\newcommand{\n}{{\mathcal N}}

\newcommand {\de} {\Delta}

\newtheorem*{facta}{Theorem A}{\bf}{\it}
\newtheorem{lemma}{Lemma}
\newtheorem{theorem}{Theorem}

\newtheorem{proposition}{Proposition}

\newcommand {\Infinity} {\infinity}

\begin{document}
\title{Coefficients multipliers of weighted spaces  of harmonic functions} 
\author{Kjersti Solberg Eikrem and Eugenia Malinnikova}
\email{kjersti.eikrem@gmail.com, eugenia@math.ntnu.no}
\thanks{The second author is supported by Project 213638 of the Research Council of Norway.}

\begin{abstract}
Let $h_g^\infty$ be the space of harmonic functions in the unit ball  that are bounded by some increasing radial function $g(r)$ with $\lim_{r\rightarrow 1} g(r)=+\infty$; these spaces are called growth spaces. We describe functions in growth spaces by the Ces\`aro means of their expansions in harmonic polynomials and apply this characterization to study coefficient  multipliers between growth spaces. Further, we introduce spaces of harmonic functions of regular growth and show that oscillation operators considered in \cite{LM} can be realized as multipliers mapping growth spaces to corresponding spaces of regular growth. 
\end {abstract}
\maketitle

\section {Introduction}

\subsection {Weighted spaces}
Let $g$ be a positive increasing continuous function such that  $\lim_{r\rightarrow +\infty} g(r)=+\infty$ and $g(1)=1$, we say that $g$ is a weight function. We study weighted (or growth) spaces of harmonic functions defined by
\begin{equation*}
h^\infty_g =\{u:\B\rightarrow\R, \Delta u=0, |u(x)|\le Kg\left(1/(1-|x|)\right)\ {\text{for some}}\ K>0\},
\end{equation*}
where $\B$ is the unit ball in $\R^{N+1}$. For $1\le p<+\infty$ we also define
\[
h^p_g=\{u:\B\rightarrow\R, \Delta u=0, \left(\int_{S}|u(rx)|^pds(x)\right)^{1/p}\le Kg\left(1/(1-r)\right)\}.\]
We assume always that
the weight function $g$ satisfies the following doubling condition
\begin{equation}
\label{gdouble}
g(2x)\le Dg(x).
\end{equation}
Our main examples are  $g (x)=x^\alpha $ and $g(x)=\left(1 +\log x\right)^\alpha $ for $\alpha>0 $.

In the present work we study operators between various growth spaces, first
we characterize functions in the weighted spaces in terms of the Ces\`{a}ro and de la Vall\'ee-Poussin sums and of their Fourier series. We show that $u\in h^p_g$, $1\le p\le \infty$, if and only if $\|\sigma_n^du\|_p\le Cg(n)$, where $d>(N-1)/2$ and $C$ does not depend on $n$. The results on Ces\'aro sums in dimension two can be found in \cite{BST}(see also references therein)  and for more general spaces in \cite{P1,P2}. The de la Vall\'ee-Poussin sums for growth spaces in dimension two were used in \cite{L1} to describe isomorphism classes of growth spaces. Our motivation comes from the study of the oscillation integral for harmonic functions in growth spaces, see \cite{LM, E, EMM}, we also work in higher dimensions and use the standard results on Ces\`{a}ro sums of spherical harmonics, \cite{S,BC,DX}, and some properties of the doubling weights.  The characterization is applied to describe the coefficient multipliers on the weighted spaces generalizing some results of \cite{SW}. Further we define new spaces, so called spaces of regular growth and show that coefficient multipliers that correspond to averaging radial operator map growth spaces to these new spaces of harmonic functions. 

The weighted spaces of harmonic functions in the unit disk have been studied by A.~L.~Shields and D.~L.~Williams in  \cite{SW}, G.~Bennett, D.~A.~Stegenga and R.~M.~Timoney in \cite {BST},  M.~Pavlovi\'c in \cite{P1,P2}, W.~Lusky in \cite {L1,L}, and others. 
In higher dimensions, the corresponding spaces of harmonic functions on the half spaces of $\R^{N+1}$ were recently characterized through wavelet expansions of the  boundary values (represented by  distributions on the boundary plane) in \cite{EMM}. The spaces of distributions in question are generalizations of the Besov spaces $B^{-s, \infty}_\infty$, the latter appear for the particular choice of the weight $g(x)=x^{s}$. In this work we give a different description of the corresponding spaces of distributions on the unit sphere.

Literature on similar spaces of analytic functions, weighted Hardy  spaces, is  extensive. We mention just some works with technique and ideas that are close to our presentation, \cite{SW1,AS,Vu,Bl,No,BRV,GPP,Do}. In particular, functions in weighted Hardy spaces $H^p_g$ with doubling weights have been characterized by the partial sums of their Fourier series for $1<p<\infty$ in a recent article by E.~Doubtsov, \cite{Do}, see also \cite{P1,P2}. This result is based on the boundedness of Fourier block projections that does not hold for $p=1,\infty$ or  $N>1$.  

\subsection {Coefficient multipliers}
Let $A $ and $B $ be
sequence spaces consisting of sequences of the form $\{a_j\}_{j \in J}$ where $a_j\in \R$ and $J$ is some index set.  A sequence $\la=\{\lambda_j\}_{j\in J} $ is called a multiplier from $A $ to $B $ if $\la a=\{\la_j a_j\} \in B $ for all $a\in A $. We consider real valued sequences and multipliers and denote the set of multipliers from $A$ to $B$
by $(A,B) $. 
 
In this work the sequences are Fourier coefficients of some functions (or distributions) on the unit sphere in $\R^{N+1}$. 
Then it is convenient to choose $J$ to be the index set for a basis of spherical harmonics. For the two-dimensional case the index set can be identified with $\Z$.
Let $\hk$ be the space of spherical harmonics of degree $k $ on the $N$-dimensional sphere and
let $\{Y_{kl}\}_{l=0}^{L_k} $ be an orthonormal basis for $\hk$ and $L_k=\dim \hk$. We denote by the same letters $Y_{kl} $ the corresponding homogeneous harmonic polynomials.
We identify a harmonic function 
\[u(x)= \sum_{j=(k,l)\in J}a_jY_{kl}(x)\] 
with the sequence $\{a_j\} $ the sequence of its coefficients in spherical harmonics expansion.  Note that for multipliers we allow distinct factors for spherical harmonics of the same order, one may  restrict the notion of multipliers only to the sequences $\lambda_j$ such that $\lambda_j=\lambda_{j'}$ when $j=(k,l)$ and $j'=(k,l')$, see for example \cite[Chapter 2]{DX} and references therein.  We  also identify the sequence $\lambda$ with a formal series
\[\lambda (x,y)= \sum_{j=(k,l)\in J}\lambda_jY_{kl}(x)Y_{kl}(y);\]
when the series converges  we get a function on  $\B\times\B$ harmonic in each variable. 
This series corresponds to the multiplier operator and  does not depend on the choice of  basis.

\subsection{Multipliers of spaces of analytic and harmonic functions} A number of natural operators on holomorphic and harmonic functions can be considered as Fourier coefficient multipliers.  Multipliers between various spaces of holomorphic functions have been studied by many authors; see for example \cite {Vu}, \cite {No} and \cite {BRV} for multipliers of
Hardy and Bergman spaces and \cite {Bl} for multipliers of more general spaces. J.~M. Anderson and A.~L.~Shields have described multipliers between the Bloch space and $\ell^p $ in \cite {AS}. The case of spaces of harmonic functions is different, for the classical unweighted space it corresponds to the difference between multipliers of $L^p$ and $H^p$. Multipliers in weighted spaces $\hv$ of harmonic functions were  studied in dimension two.
G. Bennett, D. A. Stegenga and R.~M.~Timoney determined the multipliers from $\hv $ to $\ell^p $ in \cite {BST} in the case where $g$ grows fast, for example when $g (x) =x^\alpha $ 
for $\alpha  >0 $, and this result was generalized to other weights with doubling in \cite{Ethesis}.  
It was proved by A.~L.~Shields and  D.~L.~Williams in \cite {SW} that  for the two-dimensional case and under some regularity assumptions on $g$ a sequence  $\lambda $ is a multiplier from $\hg$  to itself if and only if 
\begin {equation}
\label {m1}
\int_{0}^{2\pi}\left|\sum_{j=-\infinity}^{\infty} g(|j|)\lambda_jr^{|j|}e^{ij\ta} 
\right|d\ta \le C g\left(\fraction {1} {1 -r}\right).
\end {equation}

In this article we  give another proof of this result for weighted spaces of  harmonic functions in the unit ball of $\R^{N+1}$, where $N\ge 1$. We also show that for doubling weights  the regularity conditions are not essential, one can always replace the weight $g$ by a regular one without changing the corresponding space $h_g^\infty$.  Then we discuss how the space of multipliers depends on the weight $g$ and give a number of examples.  We also give a description of multipliers between some pairs of weighted spaces and show that weighted averaging along radii, studied in \cite{LM,E} for the case $N=1$  can be considered as a multiplier with $\lambda_{kl}\asymp g(k)^{-1}$. One of our main results, Theorem \ref{th:main}, gives a new description of the image of $h_g^p$ under the action of this multiplier.

\section {Characterization of functions in weighted spaces}
\subsection{Ces\`aro means}
We consider the standard difference operators on the space of sequences.
Let $\mathbf{b}=\{b_k\}_{k=0}^\infty $ be a sequence, then $\Delta\mathbf{b}$ is a new sequence defined by
\[(\Delta \mathbf{b})_k =b_k -b_{k +1} ,\]
 For $l=1,2,..$ we further define
 $\de^l \mathbf{b}=\de^{l -1} (\de \mathbf{b}),$
where $\Delta^1=\Delta$ and $\Delta^0=I$ is the identity operator.
Then
\[\de^l b_k =(\de^l\mathbf{b})_k=\sum_{j=0}^l (-1)^j {l \choose j}b_{k + j}. \]
For $m\in \R $ the Ces\`aro $(C,m) $ means of the sequence $\mathbf b=\{ b_k\} $ are defined by 
\[s_n^m (\mathbf{b})=\fraction1 {A_n^m}\sum_{k =0}^n A_{n -k}^m b_k, \]
where $A_k^m = {k +m\choose k} =(k+m)(k+m-1)...(m+1)/(k!)$. We have also $A_k^m\le C_m k^m$ for $k\in\N$. 

For $m =1,2,...$ the summation by parts formula is given by
\begin {equation}\label {summation by parts}
\sum_{k=0}^\infinity a_kb_k =\sum_{k=0}^\infinity \de^{m +1}a_{k}\sum_{j=0}^k A_{k -j}^{m}b_j =\sum_{k=0}^\infinity (\de^{m +1}a_{k}) A_{k}^{m}s_k^m(\mathbf{b}),
\end {equation}
if we assume that $\lim_{k \rightarrow \infinity} (\Delta^l a_k)A_k^{l}s_k^l (\mathbf{b})=0$ for $l =0,1,...,m $,
see for example Appendix~A.4 in \cite {DX}.

For each $k=0,1,..$ the zonal harmonic of order $k$ is defined by
\[Z_k (x,y) 
 =\sum_{l=0}^{\dim \hk} Y_{kl} (x)Y_{kl} (y),\]
and this function does not depend on the choice of  basis. The value of $Z_k(x,y)$  depends on $\langle x,y\rangle $ only, and we write $h_k (\langle x,y  \rangle )=Z_k (x,y)  $. Similarly to the usual Ces\`aro means in dimension two, we consider higher order Ces\`aro means of the expansions of harmonic polynomials in higher dimensions. An interesting observation due to Kogbetliantz \cite{Kog} is that the Ces\`aro $(C,m)$ means of the zonal harmonics are positive for $m\ge N$, see also  \cite [Theorem 2.4.3]{DX}. We define
$W_k^{m} (x,y)$ to be the Ces\`aro $(C,m) $ means of $ Z_k (x,y)$, i.e. 
\[W_k^{m} (x,y) =\fraction1 {A_k^m}\sum_{j =0}^k A_{k -j}^m Z_j ( x,y) . \]
 For $m\ge N $ we have $W_k^m( x,y)\ge 0$ and this yields
\begin{equation*}
\|W_k^m(\cdot, y)\|_1=\int_S|W_k^m(x,y)|ds(x)=\int_S W_k^m(x,y)ds(x)=\int_S Z_0(x,y)ds (x)=1.
\end{equation*} 
 for each $k$. Here and in what follows, the norms $\|\cdot\|_p$ are with respect to the normalized surface measure on the unit sphere of $\R^{N+1}$.  Moreover, $\sup_y\|W_k^m(\cdot, y)\|_1\le C$ when $m>(N-1)/2$, see \cite[Theorem 2.4.4]{DX} and references therein, in particular \cite{S,BC}. For the rest of the text we fix a positive integer $d$ such that $d>(N-1)/2$ (when $N=1$ one may choose $d=1$).

We will also need some auxiliary harmonic polynomials.  Let $a_m=\left(1-\fraction1m\right)^{-m},$ $ m\ge 0$, then   $a_m$ are uniformly bounded. Consider the function $q_m(t)=a_m^{t}$ on $0\le t\le 1 $ and define a continuation $q_m$ of this function for $t >1 $ such that $q_m(t) = 0$ for $t\ge 2 $ and the derivatives of $q_m$ up to order $d+1$ are bounded by some uniform constant $M=M(d)$ that does not depend on $m$. Let also $q_{m,n,k} =q_m (\fraction {k}n) $. Now define 
\begin{equation} \label {Q}
Q_{m,n} ( x,y ) =\sum_{k=0}^{2n} q_{m,n,k} Z_k (x,y) .
\end{equation}

\begin {lemma}
\label {q several}
There exists a constant  $C $ such that the functions  
$Q_{m,n} $ defined by \eqref{Q}
satisfy $\|Q_{m,n}(x,\cdot)\|_1=\|Q_{m,n}(\cdot,y) \|_1\le C$ 
 for all $m$ and $n$. 
\end {lemma}

\begin {proof} 
By \eqref {summation by parts} we get
\begin{equation*}
Q_{m,n} (x,y) =\sum_{k=0}^\infinity q_{m,n,k} Z_k (x,y)=\sum_{k=0}^\infinity \de^{d+1} q_{m,n,k} A_k^{d}W_k^{d}(x,y).
\end{equation*}
Since the functions $q_{m}$ defined above have bounded derivatives up to order $d+1$ and vanish for $t\ge 2$, we have $|\de^{d+1} q_{m,n,k} | \le C\fraction1{n^{d+1} }$, and $\de^{d+1} q_{m,n,k}=0$ for $k>2n$. Then
\begin{equation*}
\|Q_{m,n} ( \cdot ,y )\|_1 \le \sum_{k=0}^\infinity |\de^{d+1} q_{m,n,k} |A_k^{d}\|W_k^{d}( \cdot ,y )\|_1\\
\le C \sum_{k=0}^ {2n}\fraction{A_k^d}{n^{d+1}}\le C \sum_{k=0}^{2n}\fraction1 {n\cdot d!}\le C.
\end {equation*}
\end {proof}
We will be interested in two particular cases of these function, 
\begin{equation}\label{QR}
Q_n(x,y)=Q_{n,n}(x,y)\quad {\text{and}}\quad R_n(x,y)=Q_{0,n}(x,y).
\end{equation}
For a harmonic function $u$ in the unit ball we also define the operators
\begin{equation}\label{eq:R}
\mathcal{R}_nu(rx)=\int_S R_n(x,y)u(ry)ds(y),\quad |x|=1,\  0\le r<1.
\end{equation}
The definition depends on the smooth extension of the constant function $q_0$ from $[0,1]$ to $[0,2]$. In dimension two we may choose this extension such that the resulting operators are the de la Vall\'ee-Poussin sums of $u$. These operators in connection with growth spaces were used by W. Lusky in \cite{L1}.
\subsection{Characterization}
First we prove that the doubling condition implies a useful estimate. We note also that a milder restriction on the growth of the weight function than doubling would work here.

\begin{proposition} 
\label {estimate with A}
Let $g$ fulfill \eqref {gdouble}, then 
for any $m\ge 0$ we have 
$$\sum_{k=0}^\infty r^kA_k^mg(k)\le C_{m,D}g\left(\frac1{1-r}\right)(1-r)^{-m-1}.$$
\end{proposition}

For $m=1$ this inequality can be found in \cite{SW}.

\begin {proof}
Let $r $ be given and choose $N $ such that $1 -\fraction1 {N} <r\le1 -\fraction1 {N+1} $. Then
\[
\sum_{k=0}^{N} r^kA_k^mg(k)\le g (N) A_{N}^{m}\sum_{k=0}^{N } r^k\le C_m g (N) N^{m}\frac1{1-r} 
\le C_mg\left(\frac1{1-r}\right)(1-r)^{-m-1}
\]
and the rest $\sum_{k =N+1 }^\infinity r^kA_k^mg(k)$ can be written as
\begin {multline*}
\sum_{j =0}^\infinity r^ {2^jN}\sum_{i = 1}^{2^jN} r^iA_{2^jN +i}^mg(2^jN +i)
\le \frac{C_m}{1-r}\sum_{j =0}^\infinity r^ {2^jN}g(2^{j+1}N)(2^{j+1}N)^m\\
\le \frac{C_m N^mg (N)}{1-r}\sum_{j =0}^\infinity e^{-2^j}D^{j+1}2^{(j+1)m}
\le C_{m,D} g\left(\frac1{1-r}\right)(1-r)^{-m-1},
\end {multline*}
the required inequality follows.
\end {proof}

Let $u$ be  a harmonic function in the unit ball,  $u(x) =\sum_{j=0}^\infty \sum_{l=1}^ {L_j} a_{jl}Y_{jl}(x) $. The 
Ces\`aro $(C,m) $ means of $u $ are denoted by 
\[\sigma_n^m u (x)=\fraction1 {A_n^m}\sum_{j =0}^n A_{n -j}^m \sum_{l=1}^ {L_j}a_{jl}Y_{jl}(x). \]
For a polynomial of the form $P_n (x,y) =\sum_{j=0}^{n}\sum_{l=0}^{L_j} c_{jl} Y_{jl}(x)Y_{jl}(y)$, we write
\[u*P_n(x)=\int_{S} P_n (x',y)u (ry)ds(y)=\sum_{j=0}^\infty\sum_{l=1}^{L_j}a_{jl}c_{jl}Y_{jl}(x),\]
where  $x=rx', x'\in\s^N$.
In particular, $\sigma_n^m u=u*W_n^m$.

In \cite {EMM} we showed that harmonic functions in the upper half-space of $\R^{N+1} $ that are bounded by a majorant can be characterized by the size of their convolutions with functions having compactly supported Fourier transforms. For functions on the ball this corresponds to estimates of the convolutions with polynomials. The statement below ((a) equivalent to (c)) also generalizes the result for weighted  spaces of harmonic functions on the unit disk proved in \cite{BST}, to weighted spaces in several dimensions.

\begin {theorem}
\label {restriction d}
Let $u(x) =\sum_{j=0}^\infty \sum_{l=0}^ {L_j} a_{jl}Y_{jl}(x) $ be a harmonic function on on the unit ball of $\R^{N+1}$, $1\le p\le \infty$, and let $g$ be a weight function satisfying the doubling condition and $d>(N-1)/2$. Then the following are equivalent:\\
(a) $u\in h_g^p $, \\ 
(b) There exists a constant $C $ such that
$\|u*P_n\|_p\le C \sup_{x}\|P_n(x,\cdot)\|_1 g(n)$
for   any polynomial of the form $P_n (x,y) =\sum_{j=0}^{n}\sum_{l=0}^{L_j} c_{jl} Y_{jl}(x)Y_{jl}(y)$, \\ 
(c) $\|\sigma_n^{d}u\|_p\le Cg (n) $,\\
(d) $\|\mathcal{R}_nu\|_p\le Cg(n)$.
\end {theorem}

\begin{proof} First, we show that (a) implies (b).
Suppose that $P_n$ is a polynomial as in (b), and 
let $Q_n=Q_{n,n} $ be the function in Lemma~\ref{q several}. Let also $t=1 -1/n $, then $q_{n,j} =t^{-j} $ for $j\le n $.
 We consider a new polynomial
\[
\tilde{P}_n(x,y)= \sum_{k=0}^n \sum_{m=1}^ {L_k}
  c_{km} t^{-k}Y_{km} (x)Y_{km}(y).\]
Then, applying Lemma \ref{q several}, we obtain
\begin{multline*}
\sup_x\|\tilde{P}_n(x,\cdot)\|_1=\sup_x\int_{S}\left| \int_{S}\sum_{j=0}^n \sum_{l=1}^ {L_j} 
c_{jl}Y_{jl} (x)Y_{jl}(z) \sum_{k=0}^{\infty} 
q_{n,k}Z_k (z,y) ds(z)\right|ds(y)\\=\sup_x\int_S\left|\int_S P_n(x,z)Q_n(z,y)ds(z)\right|ds(y) 
\le C \sup_x\|P_n(x,\cdot)\|_1.
\end{multline*}

Now let $u\in h^p_g$, then using $S_{t}=\sum_{k=0}^\infty t^kZ_k$, we get
\begin {multline*} 
\|u*P_n\|_p^p=\|u*S_t*\tilde{P}_n\|_p^p\le \\
\le  C \left(g \left(\fraction1 {1 -t} \right)\right)^p\sup_x\|\tilde{P}_n(x,\cdot)\|_1^p\le Cg(n)^p\sup_x\|P_n(x,\cdot)\|_1^p.
\end {multline*}
Hence (b) follows.

It is clear that  (b) implies (c) and (d); we take convolutions with $W_n^{d}(x,y)$ and $R_n(x,y)$ (see \eqref{QR} and Lemma \ref{q several})
and we are done since the $L^1$-norms of $W_n^{d}(x,\cdot) $ and $R_n(x,\cdot)$ 
are bounded uniformly on $n $ and $x\in\Sm$.

To finish the proof we show that (c) implies (a) and (d) implies (c).
First assume  (c). By \eqref{summation by parts},
\begin {equation*} 
u(rx) =\sum_{j=0}^\infty \sum_{l=0}^ {L_j} a_{jl}r^{j}Y_{jl}(x) = (1  -r)^{d+1 }\sum_{k =0}^\infinity r^kA_k^{d}\sigma_k^{d}u (x).
\end {equation*}
Then (a) and  Proposition \ref {estimate with A} imply
$\|u(rx)\|_p \le (1  -r)^{d+1}\sum_{k =0}^\infinity r^kA_k^{d}g (k)\le C g (\fraction1 {1 -r})$, and then $u\in h_v^p $.

Finally, assume that (d) holds. Then $u*P_n=\mathcal R_n u*P_n$ and (b) follows since $\|u*P_n\|_p\le \|\mathcal R_n u\|_p\sup_x\|P_n(x,\cdot)\|_{L^1}.$

\end {proof}

For the case $N=1$ and $1<p<+\infty$ (a) is equivalent to (c) with $d=0$, see \cite{Do}. This does not hold for $p=+\infty$, see \cite{BST}, the choice of optimal $d$ for particular $p$ is a delicate question, see \cite[Chapter 2.5]{DX} for a related discussion.

\section {Regular multipliers}
\label {M}
\subsection{Some multipliers from $h_g^p$  to $h_{\gt}^p$} Let $f $ be a positive, increasing 
function on $[1,\infinity)
 $ which has derivatives up to order $d$ and satisfies \eqref{gdouble}. 
We also assume that there exists a constant $C >0 $ such that 
\begin {equation}\label {regularity}
|f^{(m)} (t) |\le C \fraction {f^{'} (t)} {t^{m -1}} 
\end {equation}
for $1\le m\le d+1$.
This assumption concerns the regularity of $f $ and does not restrict the growth further. Our usual examples $f (x)=x^\alpha $ and $f(x)=\left(1+\log x\right)^\alpha$ for $\alpha>0 $ fulfill \eqref{regularity}. In the next subsection we also show that for  any weight $g$ satisfying  the doubling condition one can construct an equivalent weight $f$ such that $f(n)\asymp g(n)$ and $f$ satisfies \eqref{regularity}.
Let $u(x) =\sum_{j=0}^\infty \sum_{l=1}^ {L_j} a_{jl}Y_{jl}(x) $  and
define an operator $H_f $ by
$$H_f u (x) =\sum_{j=0}^\infty \sum_{l=1}^ {L_j} f(j)a_{jl}Y_{jl}(x).$$ 
This operator can also be considered as an operator on sequences by the identification of functions with their coefficients. 
We also define an operator $H_f^{-1}$ by
$$H_f^{-1} u (x) =\sum_{j=0}^\infty \sum_{l=1}^ {L_j} \fraction {a_{jl}} {f (j)}Y_{jl}(x). $$ 
We will now prove the following:

\begin {theorem}\label{multiplication}
Assume $f $ and $g $ satisfy \eqref{gdouble} 
and $f $ satisfies 
\eqref {regularity}, $1\le p\le\infty$. \\ 
(a) Let  
$u \in h^p_g $,  
	then 	$H_fu\in h^p_{fg}$. \\
(b) Assume there exists $\ve >0 $ 
such that $g(n)/f^{1 +\ve}(n) \nearrow\infinity $  
when $n\rightarrow\infinity $ and let $u \in h^p_g $,  
	then   
$H_f^{-1}u\in h^p_{g/f} $. \\ 
(c) If $g/f $ is non-decreasing and $u\in h^p_g $, then $H_f^{-1}u\in h^p_{(g/f)\log f} $.  
\end {theorem}

In \cite {SW} parts (a) and (b) were proved in the two-dimensional case.
The proof given there is based on a duality construction which involves the measure whose moments coincide with $f(n)$. Our proof is different and employs
Theorem ~\ref {restriction d}.

\begin {proof}
(a) By Theorem \ref {restriction d} and \eqref {summation by parts},
\begin {multline*}
\|\sigma_n^{d} (H_fu)\|_p\le Cg (n)\left\|\fraction 1 {A_n^{d}}\sum_{j=0}^{n} A_{n -j}^{d}f (j)Z_j (x,\cdot)\right\|_1\\
\le Cg (n)\left\|\fraction 1 {A_n^{d}}\sum_{j=0}^{n} \Delta^{d+1} ( A_{n -j}^{d}f (j))\sum_{k =0}^j A_{j -k}^{d} Z_k (x,\cdot)\right\|_1\\
\le Cg (n)\sum_{j=0}^{n} \fraction{\Delta^{d+1} ( A_{n -j}^{d}f (j))A_j^{d}}{A_n^{d}}\left\|W_j^{d}(x,\cdot)\right\|_1
\le Cg (n)\sum_{j=0}^{n} \fraction {\Delta^{d+1} ( A_{n -j}^{d}f (j))A_j^{d}}{A_n^{d}},
\end {multline*}
so if we show that 
\begin {equation}\label {fn}
\sum_{j=0}^{n}\fraction { \Delta^{d+1} ( A_{n -j}^{d}f (j))A_j^{d}}{{A_n^{d}}}\le Cf (n) ,
\end {equation}
 we are done since Theorem \ref {restriction d} 
 then implies that $H_fu\in h^p_{gf} $.

 Define a function $h_n$ by
$$h_n (t)=M_{n,d} (t)f (t) =\left(1 -\fraction t {n +d}\right)\left (1 -\fraction t {n+ d -1}\right) ...\left (1 -\fraction t {n +1}\right)f (t)$$
for $0\le t\le n +d +1 $.
Then $h_n (j) =A_{n -j}^{d}\left(A_{n}^{d}\right)^{-1}f (j) $.
Further, it is easy to check that for $t\le n +d+1 $ we have  $|M_{n,d}^{(m)} (t)|\le C_{m,d}\, n^{-m} $ when $m\le d $ and $M_{n,d}^{(d+1)} (t)=0$.
Then, applying
\eqref {regularity}, we get 
\begin{multline}\label {derivative h}
|h_n^{(d+1)} (t)| \le\sum_{j=0}^{d+1} {d+1\choose j} |M_{n,d}^{(j)} (t)||f^{(d+1 -j)} (t)|\\
\le C_{d}\sum_{j=1}^{d} {d+1\choose j}\fraction1 {n^j}\fraction {f^{'} (t)} {t^{d -j}} \le C_{d}\fraction {f'(t)} {t^{d}} 
\end{multline}
for $t\le n +d +1$.
We note that $\Delta^{d+1}  (A_{n -j}^{d}f (j)) /A_n^{d}
$ is equal to the divided differences $h_n [j,j +1,...,j +d+1]$ of the function $h_n$ and express it on the  Peano form,
\[
\Delta^{d+1} (A_{n -j}^{d}f (j)) /A_n^{d } =h_n [j,j +1,...,j +d+1] =\fraction1 {(d+1)!}\int_j^{j +d+1} h_n^{(d+1)}(t)B_{d} (t)d t,
\]
for some B-spline $B_{d} $.
Then by \eqref{derivative h},
\begin {multline*}
|\Delta^{d+1} (A_{n -j}^{d}f (j)) /A_n^{d} | 
\le \fraction1 {(d+1)!}\int_j^{j +d+1}| h_n^{(d+1)}(t)| |B_{d} (t)|d t\\
\le C\int_j^{j +d+1}\fraction {f'(t)} {t^{d}}d t
\le C\fraction {f(j +d+1)-f (j)} {j^{d}}.
\end {multline*}
Hence
\[
\sum_{j=0}^{n}\fraction { \Delta^{d+1} ( A_{n -j}^{d}f (j))A_j^{d}}{{A_n^{d}}}
\le C \sum_{j=0}^{n}j^{d} \left(\fraction {f(j +d+1)-f (j)} {j^{d}}\right)\le Cf (n)
\]
and \eqref {fn} is proved.

(b)
 Applying Theorem \ref {restriction d} and \eqref {summation by parts}, we obtain similarly to (a)
\begin {multline*}
\|\sigma_n^{d} (H^{-1}_fu)\|_p
=\left\|\fraction 1 {A_n^{d}}\sum_{j=0}^{n} \Delta^{d+1} \left( \fraction {A_{n -j}^{d}} {f (j)}\right)\sum_{k =0}^j A_{j -k}^{d}\sum_{l=1}^ {L_k} a_{jl}Y_{jl} \right\|_p\\
\le \fraction 1 {A_n^{d}}\sum_{j=0}^{n} \Delta^{d+1} \left( \fraction {A_{n -j}^{d}} {f (j)}\right)A_j^{d}\left\|\sigma_j^{d}u \right\|_p
\le C\fraction 1 {A_n^{d}}\sum_{j=0}^{n} \Delta^{d+1} \left( \fraction {A_{n -j}^{d}} {f (j)}\right) A_j^{d}g (j).
\end {multline*}
Similarly to the proof of (a) we define a function $p_n (t) =\fraction {M_{n,d} (t)} {f (t)} $ and  estimate  the corresponding coefficients $\fraction 1 {A_n^{d}}\Delta^{d+1} \left( \fraction {A_{n -j}^{d}} {f (j)} \right)$. It can be shown, by using \eqref{regularity}, that 
$$\left|\left (\fraction {1} {f (t)}\right)^{(m)}\right|\le C \fraction {f' (t)} {t^{m -1}f (t)^2} , $$
where the constant also depends on the constant in  \eqref{regularity}.
 Then applying the estimates for $|M_{n,j}^{(j)} (t)| $ given in the proof of (a), we get for $t\le n +d+1$
\[
|p_n^{(d+1)} (t)|
\le C_{d}\sum_{j=1}^{d} {d+1\choose j}\fraction1 {n^j}\fraction {f^{'} (t)} {t^{d -j}f (t)^2} \le C_{d}\fraction {f'(t)} {t^{d}f (t)^2} .
\] 
And, using the formula for divided differences once again, we get 
\begin{multline*}
\left|\Delta^{d+1} \left  (\fraction {A_{n -j}^{d}}{f (j)}\right) \fraction1 {A_n^{d}} \right| 
\le \fraction1 {(d+1)!}\int_j^{j +d+1}| p_n^{(d+1)}(t)| |B_{d} (t)|d t\\
\le C_d\int_j^{j +d+1} \fraction{f'(t)} {t^{d}f (t)^2}d t
\le C_d\fraction {1} {j^{d}}\left (\fraction1 {f(j)}-\fraction1 {f (j+d+1)}\right).
\end{multline*}
  Let  $n_0 =1 $ 
and for some $A>1$ 
define $n_k $ by induction as 
\beq
\label {n}
n_{k+1}=\min \{l \in \N: f(l)\ge A f(n_k)\},\eeq
then 
$A f (n_k)\le  f (n_{k+1})\le AD f (n_k)$, 
where $D $ is the constant from \eqref{gdouble}.
Let $N $ be such that $n_{N -1} <n \le n_{N}$, then applying the inequality above and monotonicity of $g/f^{1+\ve}$, we get
\begin {multline} 
\label {final part}
\sum_{j=0}^{n} \Delta^{d+1} \left(\fraction { A_{n -j}^{d}} {f (j)}\right)\fraction {A_j^{d}g (j)}{{A_n^{d}}}
\le C_d \sum_{j=0}^{n}g(j)\left (\fraction1 {f(j)}-\fraction1 {f (j+d+1)}\right)\\
\le C_d\fraction {g (n)} {f (n)^{1 +\ve}}\sum_{j=0}^{n}\frac{f(j +d+1) -f (j)}{f (j)^{1 -\ve}}\\
\le C_d\fraction {g (n)} {f (n)^{1 +\ve}}\sum_{k=0}^{N-1} \sum_{j =n_{k} +1}^{n_{k+1}} \frac{f(j +d+1) -f(j )}{f(n_k)^{1-\ve}}\\ 
\le C_{d,D}A^{1-\ve}\fraction {g (n)} {f (n)^{1 +\ve}}\sum_{k=0}^N f (n_{k+1})^{\ve} \le C_{d,D}A\fraction {g (n)} {f (n)},
\end {multline}
and we are done.

(c) The proof proceeds exactly like in (b), but we replace \eqref {final part} by the following:
\begin {multline*} 
\sum_{j=0}^{n} \Delta^{d+1} \left(\fraction { A_{n -j}^{d}} {f (j)}\right)\fraction {A_j^{d}g (j)}{{A_n^{d}}}
\le C_d\fraction {g (n)} {f (n)}\sum_{j=0}^{n}j^{d}\fraction {f (j)} {j^{d}}\left (\fraction1 {f(j)}-\fraction1 {f (j+d+1)}\right)\\
\le C_d\fraction {g (n)} {f (n)}\sum_{j=0}^{n}\frac{f(j +d+1) -f (j)}{f (j)}\le C_d\fraction {g (n)} {f (n)}\sum_{k=0}^{N-1}  \sum_{j =n_{k} +1}^{n_{k+1}} \frac{f(j +d+1) -f(j )}{f(n_k)}\\
\le C_{d,D}A\fraction {g (n)} {f (n)}N\le C_{d,D}A\fraction {g (n)\log f (n)} {f (n)\log A}.
\end {multline*} 
\end {proof}

\subsection{Sharpness of the result} We give some examples showing that part (b) of  Theorem \ref {multiplication} does not hold without the assumption $g(n)/f^{1+\ve}(n) \nearrow\infinity $ and that the weight in part (c) is sharp. We consider the case $p=\infty$.

A harmonic function on the disk $\D $ is called a Hadamard gap series if 
\begin {equation}
\label {Hadamard}
u (z)=\Re\sum_{k=1}^\infty a_{n_k}z^{n_k},\quad a_{n_k}\in\C,
\end {equation}
where $n_{k+1}>\lambda n_k$ for some $\lambda>1$. 
In \cite {E} Hadamard gap series were characterized in the following way: 

\begin {facta}
 Let $\{n_k\}_{k=1}^\infty$ be a sequence of positive integers such that $n_{k+1}\ge \lambda n_k$ for each $k$, where $\lambda>1$. Assume $g $ satisfies  (\ref{gdouble}) and let $u $ be given by \eqref {Hadamard}
 where the series converges in the unit disk.
Then 
$u\in\hg $ if and only if there exists $C$ such that 
$$\sum_{n_k\le N}|a_{n_k}|\le C\, g(N)$$
for any $N\in \N$.
\end {facta}

We use this to construct some examples.

Let $u(z) =\sum_{j =0}^\infinity 2^jz^{2^j} $, then by Theorem A, $u\in\hg $ for $g (x) =x $. Let $f (x) =x/\log x $ and
consider $H_f^{-1} u(z) =\Re\sum_{j =0}^\Infinity \fraction {2^j} {f (2^j)}z^{2^j}$. 
Then 
\begin {equation*}
\sum_{j =0}^ {J} \fraction {2^j} {f (2^j)}=\log 2 \sum_{j =0} ^J j \approx J^2 \approx (\log 2^J)^2
\end {equation*}
 hence by Theorem A, $H_f^{-1} u \in\hgt$ for $\gt =(\log x)^2$, and $H_f^{-1} u \notin\hgt$ for any 
 $\gt $ which is $o ((\log x)^2)$. This shows that $H_f^{-1} u $ grows faster than $g/f =\log x $, and the growth assumption in (b) is necessary.
 
To construct a more general example, define a sequence $\{m_k\} $ for $g $ as in \eqref {n}. The function $u (z) =\sum_{j =	1}^\infinity g (m_k)z^ {m_k} $ is in $\hg $ by Theorem A. If $f =g/\log g $, we get by a similar calculation as above that $H_f ^ {-1}u\in \hgt $ for $\gt (x) =(\log g (x))^2 $, and $H^{-1}_fu$ does not belong to any weighted space with weight which grows slower than $\tilde{g}$.

If the growth of $g $ and $f $ is even more similar, for example $f =g/\log\log g $, we get for the function $u (z) =\sum_{j =	1}^\infinity g (m_k)z^ {m_k} $ that $H_f ^ {-1}u\in \hgt $ for $\gt (x) = \log g\log\log g $. This shows that (c) is sharp, since $(g/f)\log f \approx \log g\log\log g $. When we take $f=g$ and the same function $u$ we obtain $H_g^{-1}u (z) =\sum_{k =1}^\infinity z^{m_k}$. Then Theorem A implies $H_g^{-1}u\in\hi_{\log g} $, we discuss this case in the next subsection.

\subsection{Integral operators as multipliers}
\label {integral operators}
Let $q(t)$ be an increasing weight that  satisfies the doubling condition with some constant $D$ and let $w(t)=q(1/(1-t))$, for $t<1$. We consider the following integral operator  defined on harmonic functions in the unit ball (see \cite{LM,E,EMM})
\[
I_q u(x)=\int_{1/2}^1 u(tx)d(-1/w(t)).\]
Let $u(x)=\sum_{k,l}a_{kl}Y_{kl}(x)$  be the decomposition of $u$ into spherical harmonics, then
\[
I_q u(x)=\sum_{k,l}a_{kl}\left(\int_{1/2}^1t^kd(-1/w(t))\right)Y_{kl}(x).\]
It is a coefficient multiplier of the form $H_f^{-1}$, where 
\begin{equation*}
f(n)=\left(\int_{1/2}^1 t^nd(-1/w(t))\right)^{-1}.
\end{equation*}
We have the following elementary estimate from above 
\[
f(n)\le \left(\int_{1-1/n}^1 t^nd(-1/w(t))\right)^{-1}\le 4q(n).\]
To estimate $f(n)$ from below we choose integers $m$ and $M$ such that $2^{\frac{n}{2m}}>2D$ but $m>\alpha n$, where $\alpha$ depends on $D$ only, and $2^M< m\le 2^{M+1}$. Then we get
\[
\int_{1/2}^{1-1/m}t^nd(-1/w(t))\le \sum_{j=1}^M \int_{1-2^{-j}}^{1-2^{-j-1}}t^nd(-1/w(t))\le
\sum_{j=1}^M (1-2^{-j-1})^n\frac{1}{q(2^j)}.\]
We want to estimate the last sum by a multiple of its final term. It is enough to check that it grows faster than a geometric progression. Using an elementary inequality $(1-x)/(1-2x)\ge 1+x$, when  $0<x<1/2$, and the standard estimate $(1+(2m)^{-1})^{2m}\ge 2$, we obtain
\[
\frac{(1-2^{-j-1})^n} {q(2^{j})}\frac{q(2^{j-1})}{(1-2^{-j})^{n}}\ge D^{-1}(1+2^{-j-1})^n\ge D^{-1}(1+(2m)^{-1})^n\ge D^{-1}2^{\frac{n}{2m}}>2,\]
for $j=2,...,M$.
Then 
\[f(n)^{-1}\le \int_{1/2}^{1-1/m}t^nd(-1/w(t))+\frac{1}{q(m)}\le \frac{2(1-2^{-M-1})^n}{q(2^M)}+\frac{1}{q(m)}\le \frac{C}{q(n)},
\]
where $C$ depends on $D$ only. 

For $\alpha>1$ we define also
\begin{equation}\label{f}
f(\alpha)= \left(\int_{1/2}^1 t^\alpha d(-1/w(t))\right)^{-1}.\end{equation}
To check the regularity of the function $f$ we need estimates of  the integrals
\[
j_k(\alpha)=\int_{1/2}^1 |\log t|^k t^\alpha d(-1/w(t)),\]
for $k=1,..., d$.
A calculation similar to the one above implies $j_k(\alpha)\le C\alpha j_{k-1}(\alpha)$ when $k=1,..., d$ and $C=C(d,D)$.
Further, $f'=j_1 f^2$ and taking the derivatives on both sides and using induction we see that $f^{(k)}(\alpha)\le Cf'(\alpha)\alpha^{1-k}$ when $k=1,..., d$. 

We summarize the estimates of this subsection in the following statement.

\begin{proposition}\label{p:reg}
Let $q$ be a positive continuous increasing function that satisfies the doubling condition, define $w(t)=q((1-t)^{-1})$ and let $f$ be given by \eqref{f}. Then  $f$ satisfies \eqref{regularity} and there exists  $A=A(q)$  that depends on $q$ only such that $A^{-1}q(n)\le f(n)\le Aq(n)$.   
\end{proposition}

We say that $f$ is the regularization of $q$. Further,  note that $I_q=H_{f}^{-1}$ and we can apply Theorem \ref{multiplication} to this integral operator. In \cite{LM,E,EMM} the operator $I_q$ was studied on $h^\infty_q$, which corresponds to the  case $f=g$ in Theorem \ref{multiplication} (c). A more delicate result, a version of the law of the  iterated logarithm, was obtained for this case. 
In the next section we define new spaces that are better suited for the study of $I_q$ and multipliers described in part (c) of Theorem \ref{multiplication}.
\subsection{Regular growth spaces}
Let $\mathcal N=\{n_k\}$ be an increasing sequence with $n_{k+1}>2n_k$, $n_0=0$. We introduce spaces of harmonic functions of regular growth with respect to $\mathcal N$ in the following way
 \[
h^p_{\mathcal N}=\{u:\B\rightarrow\R, \Delta u=0, \sup_k\sup_{n_k\le m\le n_{k+1}}\|(\mathcal R_{m}-\mathcal R_{n_k})u\|_p<+\infty\}.\] This definition formally depends on the construction of $\mathcal R_n$, which is based on some smooth extension $q_0$ of the characteristic function of $[0,1]$. The theorem below shows however that $h^p_\mathcal N$ has another description and does not depend on the choice of $q_0$.

Let $f$ be a doubling weight and let $\mathcal N$ be a sequence associated with $f$ by \eqref{n}. Then Theorem \ref{restriction d} (d) implies that $u\in h^p_f$ if and only if 
\[\sup_{n_k\le m\le n_{k+1}}\|(\mathcal R_{m}-\mathcal R_{n_k})u\|_p\le Cf(n_{k+1})\] and  $h^p_{\mathcal N}$ is a subspace of $h^p_{\log f}$. It is easy to see that it is a proper subspace. We claim that $h^p_{\mathcal N}$ is the correct target space for the multiplier $H_f^{-1}$.   
\begin{theorem}\label{th:main}
Assume $f $  satisfies \eqref{gdouble} 
and
\eqref {regularity}, $1\le p\le\infty$, and let $\mathcal N$ be associated with $f$. Then $H_f$ maps $h_{\mathcal N}^p$ into $h_f^p$ and its inverse $H_f^{-1}$ maps $h_f^p$ into $h^p_{\mathcal N}$.
\end{theorem}
\begin{proof}
We need the following statement, if $f$ satisfies \eqref{gdouble} 
and
\eqref {regularity}, and then for any $n$  there exists  polynomial $T_{n,f}(x,y)=\sum_j a_j Z_j(x,y)$  such that $a_j=f(j)$ when $j\le n$ and and $\|T_{n,f}(x,\cdot)\|_1\le Cf(n)$. The proof is similar to one we gave when proving Theorem \ref{restriction d}(a) and (b). Let $a$ be a function with $d+1$ bounded derivatives such that $a=1$ on $[0,1]$ and $a=0$ on $(2,+\infty)$. Define $a_j=f(j)a(j/n)$, since $|(f(t)a(t/n))^{(d+1)}|\le C_k(f(t)n^{-d-1}+f'(t)t^{-d})$ when $t<2n$, we have 
\[\Delta^{d+1}a_j\le C_d (f(2n)n^{-d-1}+(f(j+d+1)-f(j))j^{-d}).\]
Then $\| T_{n,f}(x,\cdot)\|_1\le \sum_{j=1}^{2n} A_j^{d}|\Delta^{d+1}a_j|\le C_df(2n)$.

 Suppose that $u\in h_{\mathcal N}^p$ then for $n_k\le m\le n_{k+1}$ we have $(\mathcal R_{m}-\mathcal R_{n_k})(H_fu)=T_{2m,f}*(\mathcal R_{m}-\mathcal R_{n_k})u$ and 
\[\|(\mathcal R_{m}-\mathcal R_{n_k})(H_fu)\|_p\le Cf(4m)\|(\mathcal R_{m}-\mathcal R_{n_k})u\|_p\le CD^2f(n_{k+1}).\]
Thus $H_fu\in h_f^p$.

To prove the inverse we need another auxiliary function 
 $S_{n,m,f}=\sum_j b_jZ_j(x,y)$ where $b_j=f^{-1}(j)$ when $n\le j\le m$. Let function $b$ with bounded derivatives up to order $d+1$be such that  $b=1$ on (0,1) and $b=0$ on $(2,+\infty)$. Consider $b_j=(b(j/m)-b(2j/n)/f(j)$. We have $|b^{(d+1)}(t)|\le C(n^{-d-1}f(n)^{-1}+f'(t)f(t)^{-2}t^{-d})$ when $n/2\le t\le n$ and $|b^{(d+1)}(t)|\le C(m^{-d-1}f(n)^{-1}+f'(t)f(t)^{-2}t^{-d})$ when $n\le t\le 2m$.  Hence
\[\| S_{n,f}(x,\cdot)\|_1\le \sum_{j=n/2}^{n} A_j^{d}|\Delta^{d+1}
b_j|+\sum_{j=n}^{2m} A_j^{d}|\Delta^{d+1} b_j|\le C_df(n)^{-1}.\]
Now suppose that $u\in h_f^p$, we have $(\mathcal R_{m}-\mathcal R_{n_k})(H^{-1}_fu)=S_{n_k,2m,f}*(\mathcal R_{m}-\mathcal R_{n_k})u$ and 
\[\|(\mathcal R_{m}-\mathcal R_{n_k})(H^{-1}_fu)\|_p\le Cf(n_k)^{-1}\|(\mathcal R_{m}-\mathcal R_{n_k})u\|_p\le C.\]
Thus $H^{-1}_fu\in h_{\mathcal N}^p$.
\end{proof}
\section{Fourier multipliers on weighted spaces}
\subsection{Characterization of multipliers}
First, we give a description of all multipliers between two weighted spaces $h_g^\infty$ and $h_{\tilde{g}}^\infty$ when the weight $\tilde{g}$ does not grow much slower than $g$, for example the condition of the theorem holds for $\tilde{g}=g$. We remind that a multiplier is considered as a formal series and define the operator $H_f$ on $\lambda$ by
\[
H_f\lambda(x,y)=\sum_{k,l}\lambda_{kl}f(k)Y_{kl}(x)Y_{kl}(y).\]
Then $H_f\lambda$ is also a formal series but its partial sums are well-defined functions on $S\times S$.  

\begin{theorem}
\label {both}
Let $g $ and $\gt$ satisfy \eqref{gdouble}.
Assume that $\gt (n)/g (n)^\ve  
\nearrow \infinity$ as $n\rightarrow \infinity$ for some $\ve>0$. Then
$\lambda $ is a multiplier from $\hg $  to $\hgt $ if and only if 
\[\|(\sigma_n^{d}(H_f \lambda))(\cdot,y)\|_1\le C 
\gt (n),\]
 where $f$ is the regularization of $g$ defined in Proposition \ref{p:reg}.
\end{theorem}

This generalizes \eqref{m1}, which is 
Theorem 6 in \cite{SW}. We main job is already done; we will deduce this result from Theorem \ref{multiplication} in a way similar to that in \cite{SW}.

\begin {proof} Let $\la\in (\hg,\hgt) $. 
 By  Theorem \ref {multiplication} (a) 
$H_f $ is a bounded operator from $h^\infinity $ to $h_f^\infty=h_g^\infty $, where $h^\infinity $ is the space of bounded harmonic functions in the ball. 
Let $u=\sum_{k,l} a_{kl}Y_{kl}(x)\in h^\infinity $ be an arbitrary function with $\|u\|_\infty\le 1$. Then since $\lambda (H_fu)\in\hi_{\gt} $,  Theorem \ref {restriction d} implies, 
\begin {multline*}
\left |\int_{S} \sigma_n^{d}(H_f \lambda) (x,y) u(x)ds(x)\right|=\left|\sum_{k,l} a_{kl}f(k)\lambda_{kl}\frac{A_{n-k}^{d}}{A_n^{d}}Y_{kl}(y)\right|\\=|\sigma_n^{d}(\lambda(H_fu))(y)|  \le C\gt (n).
\end {multline*}
Hence $\|\sigma_n^{d}(H_g \lambda)(\cdot,y)\|_1 
  \le C\gt (n) $. 

Now suppose $\|\sigma_n^{d}(H_g \lambda)(\cdot,y)\|_1
  \le C\gt (n) $  
and let $v\in\hg $. Then by 
Theorem~\ref{restriction d},
\begin {multline*}
| \sigma^{d}_n(\la(H_f v) ) (y)|=
\left |\int_{\Sm}  v(x)\sigma^{d}_n(H_f \lambda)(x,y)ds(x)\right|\\ \le Cg(n) \| (\sigma_n^{d}(H_g \lambda))(\cdot,y)\|_1  
\le C_1 f(n)\gt(n).
\end {multline*}
Thus, $\lambda(H_f v)\in h^\infty_{f\gt}$ but $\lambda(v)=H_{f}^{-1}(\lambda(H_f v))$ and by
 Theorem \ref {multiplication} (b) 
$\la (u)\in \hgt$.
\end {proof}

\subsection {Solid spaces and weighted mix-norm spaces} 
In this subsection we collect some basic facts on solid spaces that will be needed to discuss various examples of multiplies. A sequence space $A$ is called {\it solid} if $a\la
 \in A $ whenever $a\in A $ and $\la\in \ell^\infty$. If $A $ is not solid, we can instead find solid spaces contained in it  or containing it. The  smallest solid space containing $A $ is denoted $S (A) $, and the largest solid space contained in $A $ is $s (A) $. The study of smallest and largest solid spaces in connection with multipliers problems was initiated in \cite{AS}, see also \cite{A,Buck1}.
 
 The space $\hv $ is not solid, and this means that functions in $\hg $ cannot be characterized in terms of the 
  absolute values of their coefficients. For the two-dimensional case it is known that
$S (\hv) $ is the space of functions whose coefficients satisfy 
\begin {equation}\label {large solid}
\left(\sum_{j=-n}^n |a_j|^2\right)^{1/2}\le C g(n),
\end {equation}
 and 
$s (\hv) $ consists of functions satisfying 
\begin {equation}\label {small solid}
\sum_{j=-n}^n |a_j|\le C g(n),
\end {equation}
 see \cite {BST}. The most difficult part is based on a theorem of  de Leeuw, Katznelson and Kahane, see \cite{dLKK}. Applying this result one easily gets a similar characterization in higher dimensions.

Let  $J=\{(k,l): k\in \Z_+, 1\le l\le L_k\}$ be the index set for the spherical harmonics expansion. Let further $\{n_m\}_m$ be an increasing sequence of positive integers, we define
$J_0  =\{(k,l)\in J:0\le k\le n_0\}$  and $J_m  =\{(k,l)\in J:n_{m-1}< k\le n_m\}$ for $m\ge 1 $.
Let
$\lambda $ be a sequence and let
$$\|\lambda\|_{p,q}=\left( \sum_{m=0}^\infty \left[\sum_{j\in J_m} |\lambda_j|^p\right]^{q/p}\right)^{1/q},$$
where $1\le p,q<\infinity $. 
If $p$ or $q$ are infinite, we replace the corresponding sum by a supremum.
The space of sequences $\{\lambda_j\}_{j\in J}$ for which $\|\lambda\|_{p,q} $ is finite is called a mixed-norm space. If $p =q $ this is just the usual $\ell^p $ space.
Multipliers between such spaces were determined by C.~N.~Kellogg in \cite{Ke} for the sequence $n_k=2^k$.
The mixed-norm can be generalized further by introducing a positive weight $f $:
$$\|\lambda\|_{p,q,f,\n}=\left( \sum_{m=0}^\infty \left[\fraction {(\sum_{j\in J_m} |\lambda_j|^p)^{1/p}} {f (n_m)}\right]^{q}\right)^{1/q}.$$
Let $\ell^{p,q}_{f,\n}$ be the set of sequences for which $\|\lambda\|_{p,q,f,\n} $ is finite; we call such spaces weighted mixed-norm spaces. Different mixed-normed spaces appeared in \cite{Buck}.

Let  $g$ satisfy the doubling condition and let $\n=\{n_m\}_m$ be defined by \eqref {n} with $g$ in place of $f$. We consider the corresponding partition $\n$ of the index set. Then \eqref{large solid} and \eqref{small solid} can be written as
\begin{equation*}
S(\hv)=l^{2,\infty}_{g,\n}, \quad s(\hv)=l^{1,\infty}_{g,\n}.
\end{equation*}

\subsection {Examples of multipliers on $\hg$} 
The sequence $\lambda_j=1$ is obviously a multiplier on $\hg$ (i.e., a multiplier from $\hg$ to $\hg$), and a sequence cannot be a multiplier on $\hg $  unless it is in $\ell^\infinity$. But not all bounded sequences are multipliers since $\hg $ is not solid. We determine the largest solid subspace of the multipliers from $\hg$ to itself.

Multipliers between the standard (without weight) mixed normed spaces were described in \cite{Ke}.  Similar results hold for 
the weighted mixed-norm spaces, see \cite{Ethesis}, in particular,
\begin{equation*}
(l^{2,\infty}_{g,\n}, l^{1,\infty}_{g,\n})=l^{2,\infty}_{\n},
\end{equation*}
where the last space in unweighted. It is easy to see that $s((A, B))=(S(A), s(B))$ thus we have the following
\begin{equation*}
s((\hg,\hg))=l^{2,\infty}_{\n}.
\end{equation*}
This space depends on the weight (through the partition $\n$) and it always contains $l^2$. For the weights $g(x)=x^\alpha$ we can take $n_k=2^k$ and we get  the standard mixed-norm spaces $l^{2,\infty}$. For $g(x)=(1+\log x)^\beta$ we get $n_k=2^{2^k}$ and the space $l^{2,\infty}_{\n}$ is smaller than $l^{2,\infty}$.
For any $\n$ the sequence  $\lambda_j =1/\sqrt {n_k}  $ for $n_{k - 1} <j\le n_k $ is in $\ell_{\n}^{2,\infinity} $. This gives some examples of multipliers.

It is proved in \cite{SW} (in dimension two) that  if $\gt/g\rightarrow \infty $ as $x\rightarrow \infty $, then $( \hv,\hv)\subset ( \hvt,\hvt)$ and  if in addition there is an integer $m>1$ such that $g^m/\gt\rightarrow \infty $ as $x\rightarrow \infty $, then $( \hv,\hv)=( \hvt,\hvt)$. Theorem \ref{multiplication} shows that the same is true in higher dimensions.  The inclusion is strict for pairs of weights of different growth, for example when $g(x)=x^\alpha$ and $\tilde g(x)=(1+\log x)^{\beta}$.  We see that if $l^{2,\infty}_{\n}\neq l^{2,\infty}_{\tilde{\n}}$ then $(\hv, \hv)\neq (\hvt, \hvt)$ since they have distinct solid parts. 

The multipliers we looked at so far are small ones. 
In \cite {BKV} a vector space of analytic functions $X $ is said to have the \textit {small multiplier property} when there exists $r >0 $ such that $\lambda =\{\lambda_j\} \in (X,X) $ for any $\lambda$ with $\lambda_j=O(j^{-r}),\ j\rightarrow\infty$. We see that all weighted spaces $\hv$ have the small multiplier property since $l^2\subset(\hv,\hv)$ (for example we can take any $r>1$ when $d=1$). Further Theorem \ref{multiplication} (a) shows that every regular bounded sequence is a multiplier. Then  any sequence $\{\lambda_{kl}=f(k)+d_{kl}\},$ where $f$ is bounded and satisfies \eqref{regularity} and  $\{d_{kl}\}\in l^{2,\infty}_{\n}$, is a multiplier from $\hv$ to $\hv$.

\section*{Acknowledgements}
This work is an extended version of the last chapter of the Ph.D. Thesis of Kjersti Solberg Eikrem. We are grateful to Dragan Vukoti\'c and Catherine B\'en\'eteau who read the Thesis prior to the defense for their useful comments and suggestions. The work  was carried out at Center of Advanced Study at the Norwegian Academy of Sciences and Letters in Oslo and the Department of Mathematical Sciences at NTNU, Trondheim, Norway, and it is a pleasure to thank the Center and Department for the support.
\bibliographystyle{plain}
\bibliography{bibmulti1}

\begin{thebibliography}{10}

\bibitem{A}
J.~M. Anderson.
\newblock {Coefficient multipliers and solid spaces.}
\newblock {\em J. Anal.}, 1:13--19, 1993.

\bibitem{AS}
J.~M. Anderson and A.~L. Shields.
\newblock {Coefficient multipliers of Bloch functions.}
\newblock {\em Trans. Amer. Math. Soc.}, 224(2):255--265, 1976.

\bibitem{BST}
G.~Bennett, D.~A. Stegenga, and R.~M. Timoney.
\newblock {Coefficients of Bloch and Lipschitz functions.}
\newblock {\em Illinois. J. Math.}, 25(3):520--531, 1981.

\bibitem{Bl}
O.~Blasco.
\newblock {Multipliers on spaces of analytic functions.}
\newblock {\em Canad. J. Math.}, 47(1):44--64, 1995.

\bibitem{BC}
A.~Bonami and J.-L. Clerc.
\newblock {Sommes de Ces\`aro et multiplicateurs des d\'eveloppements en
  harmoniques sph\'eriques}.
\newblock {\em Trans. Amer. Math. Soc.}, 183:223--262, 1973.

\bibitem{Buck}
S.~M. Buckley.
\newblock {Mixed norms and analytic function spaces}.
\newblock {\em Math. Proc. R. Ir. Acad.}, 100A(1):1--9, 2000.

\bibitem{Buck1}
S.~M. Buckley.
\newblock {Relative solidity for spaces of holomorphic functions}.
\newblock {\em Math. Proc. R. Ir. Acad.}, 104A(1):83--97, 2004.

\bibitem{BKV}
S.~M. Buckley, P.~Koskela, and D.~Vukoti\'c.
\newblock {Fractional integration, differentiation, and weighted Bergman
  spaces}.
\newblock {\em Math. Proc. Cambridge Philos. Soc.}, 126(2):369--385, 1999.

\bibitem{BRV}
S.~M. Buckley, M.~S. Ramanujan, and D.~Vukoti\'c.
\newblock {Bounded and compact multipliers between Bergman and Hardy spaces.}
\newblock {\em Integral Equations Oper. Theory}, 35(1):1--19, 1999.

\bibitem{DX}
F.~Dai and Y.~Xu.
\newblock {\em {Approximation Theory and Harmonic Analysis on Spheres and
  Balls.}}
\newblock {Springer, New York}, 2013.

\bibitem{dLKK}
K.~de~Leeuw, Y.~Katznelson, and J.-P. Kahane.
\newblock {Sur les coefficientes de Fourier des functions continues}.
\newblock {\em C.R. Acad. Sci. Paris S\'er. A-B}, 285(16):A1001--A1003, 1977.

\bibitem{Do}
E.~Doubtsov.
\newblock {Characterisation of Hardy growth spaces with doubling weights.}
\newblock {\em Bull. Aust. Math. Soc.}, to appear.

\bibitem{Ethesis}
K.~S. Eikrem.
\newblock {\em {Characterization and boundary behavior of harmonic functions in
  growth spaces}}.
\newblock PhD thesis, Norwegian University of Science and Technology (NTNU),
  2013.

\bibitem{E}
K.~S. Eikrem.
\newblock {Hadamard gap series in growth spaces}.
\newblock {\em Collect. Math.}, 64(1):1--15, 2013.

\bibitem{EMM}
K.~S. Eikrem, E.~Malinnikova, and P.~Mozolyako.
\newblock {Wavelet decomposition of harmonic functions in growth spaces}.
\newblock {\em {J. Anal. Math.}}, 122:87--111, 2014.

\bibitem{GPP}
D.~Girela, M.~Pavlovi\'c, and J.\'A. Pel\'aez.
\newblock {Spaces of analytic functions of Hardy-Bloch type}.
\newblock {\em {J. Anal. Math.}}, 100:53--81, 2006.

\bibitem{Ke}
C.~N. Kellogg.
\newblock {An extension of the Hausdorff-Young theorem.}
\newblock {\em Michigan Math. J.}, 18:121--127, 1971.

\bibitem{Kog}
E.~Kogbetliantz.
\newblock {Recherches sur la sommabilit\'e de s\'eries ultraph\'eriques par le
  m\'ethode des moyennes arithm\'etiques.}
\newblock {\em J. Math. Pure Appl. Ser. 9}, 3:107--187, 1924.

\bibitem{L1}
W.~Lusky.
\newblock {On weighted spaces of harmonic and holomorphic functions.}
\newblock {\em J. London Math. Soc.}, 51:309--320, 1995.

\bibitem{L}
W.~Lusky.
\newblock {On the isomorphism classes of weighted spaces of harmonic and
  holomorphic functions.}
\newblock {\em Studia Math.}, 175(1):19--45, 2006.

\bibitem{LM}
Yu. Lyubarskii and E.~Malinnikova.
\newblock {Radial oscillation of harmonic functions in the Korenblum space}.
\newblock {\em Bull. Lond. Math. Soc.}, 44(1):68--84, 2012.

\bibitem{No}
M.~Nowak.
\newblock {Coefficient multipliers of spaces of analytic functions.}
\newblock {\em Ann. Univ. Mariae Curie-Sk\l odowska Sect. A}, 52(1):107--119,
  1998.

\bibitem{P1}
M.~Pavlovi\'c.
\newblock {Mixed norm spaces of analytic and harmonic functions, I.}
\newblock {\em Publ. Inst. Math.}, 40 (54):117--141, 1986.

\bibitem{P2}
M.~Pavlovi\'c.
\newblock {Mixed norm spaces of analytic and harmonic functions, II.}
\newblock {\em Publ. Inst. Math.}, 41 (55):97--110, 1987.

\bibitem{SW1}
A.~L. Shields and D.~L. Williams.
\newblock {Bounded projections, duality, and multipliers in spaces of analytic
  functions.}
\newblock {\em Trans. Amer. Math. Soc.}, 162:287--302, 1971.

\bibitem{SW}
A.~L. Shields and D.~L. Williams.
\newblock {Bounded projections, duality, and multipliers in spaces of harmonic
  functions.}
\newblock {\em J. Reine Angew. Math.}, 299/300:256--279, 1978.

\bibitem{S}
G.~{Szeg\"o}.
\newblock {\em {Orthogonal polynomilas.}}
\newblock {American Mathematical Society. Providence, Rhode Island}, fourth
  edition, 1975.

\bibitem{Vu}
D.~Vukoti\'c.
\newblock {On the coefficient multipliers of Bergman spaces.}
\newblock {\em J. Lond. Math. Soc., II. Ser.}, 50(2):341--348, 1994.

\end{thebibliography}
\end{document}